\documentclass[hidelinks,onefignum,onetabnum]{siamart250211}

\usepackage{amsmath}
\usepackage{amsfonts}
\usepackage{amssymb}
\usepackage{bm}
\usepackage{graphicx}
\usepackage{algorithm}
\usepackage{algpseudocode}
\usepackage{booktabs}

\algrenewcommand\algorithmicrequire{\textbf{Input:}}
\algrenewcommand\algorithmicensure{\textbf{Output:}}

\newsiamthm{assumption}{Assumption}
\newsiamremark{remark}{Remark}

\DeclareMathOperator{\Cov}{Cov}
\DeclareMathOperator{\Var}{Var}
\DeclareMathOperator{\tr}{tr}
\DeclareMathOperator{\diag}{diag}
\DeclareMathOperator*{\argmin}{argmin}
\DeclareMathOperator{\dist}{dist}

\newcommand{\R}{\mathbb{R}}
\newcommand{\E}{\mathbb{E}}
\newcommand{\Id}{I}
\newcommand{\cF}{\mathcal{F}}

\newcommand{\cR}{\mathcal{R}}
\newcommand{\barE}{\overline E}
\newcommand{\Fhat}{\widehat F}
\newcommand{\ghat}{\widehat g}
\newcommand{\SigF}{\Sigma_F}
\newcommand{\SigH}{\Sigma_{Hv}}

\newcommand{\ip}[2]{\left\langle #1,#2\right\rangle}
\newcommand{\norm}[1]{\left\|#1\right\|}

\headers{Geometry-Preserving Saddle Search}{Y. Yu and Y. Wang}

\title{Geometry-Preserving Nudged Elastic Band and Dimer Methods under Anisotropic Force Uncertainty}

\author{Yifan Yu\thanks{Department of Mathematics, Faculty of Science, National University of Singapore, 10 Lower Kent Ridge Road, Singapore.}
\and Yangshuai Wang\thanks{Corresponding author. Department of Mathematics, Faculty of Science, National University of Singapore, 10 Lower Kent Ridge Road, Singapore
  (\email{yswang@nus.edu.sg}).}}

\ifpdf
\hypersetup{
  pdftitle={Geometry-Preserving Nudged Elastic Band and Dimer Methods under Anisotropic Force Uncertainty},
  pdfauthor={Y. Yu and Y. Wang}
}
\fi

\begin{document}

\maketitle

\begin{abstract}
The nudged elastic band (NEB) and Dimer methods are standard tools for computing minimum-energy paths and index-one saddle points in atomistic transition problems.  They are increasingly driven by surrogate or learned force models, whose force errors are often anisotropic and spatially varying near transition states and defect cores, where saddle-search iterations are most sensitive.
We introduce uncertainty-aware NEB and Dimer methods (UA-NEB, UA-Dimer) that use covariance as an optimizer-level reliability metric while preserving the mean-potential saddle-search equations: an oblique normal projection for NEB and covariance-weighted rotation and translation for Dimer.
Both algorithms fit Robbins--Monro recursions; under a local Lyapunov stability hypothesis, verified explicitly for a canonical UA-NEB setting and stated as a hypothesis for UA-Dimer, the stochastic iterations converge almost surely within the corresponding local stability neighborhood.  In the analytic benchmark, UA-NEB reduces mean barrier error by $21\%$ relative to stochastic NEB and UA-Dimer reduces the reflected-gradient residual by $22\%$; in the 127-atom tungsten-vacancy benchmark, full UA-NEB reduces mean barrier error by $56\%$ relative to stochastic NEB and by $23\%$ relative to diagonal covariance weighting.
These results show that anisotropic uncertainty is most useful when embedded in the constrained geometry of the optimizer rather than collapsed into a scalar acquisition or trust criterion.
\end{abstract}



\section{Introduction}
\label{sec:introduction}

Atomistic transition rates depend exponentially on index-one saddle barriers~\cite{vineyard1957frequency}: an error of a few millielectronvolts can shift a predicted rate by tens of percent.  Such barriers govern activated events in crystals, molecules, and catalysts, making accurate saddle search central to predictive kinetics.  The nudged elastic band (NEB) method~\cite{jonsson1998neb,henkelman2000improved,henkelman2000climbing} and Dimer method~\cite{henkelman1999dimer,gould2016dimer} are standard tools for this task.  In large-scale searches, however, the optimizer often queries stochastic, ensemble, surrogate, or learned forces rather than exact deterministic forces.  Machine-learned interatomic potentials (MLIPs)~\cite{bartok2010gaussian,batatia2022mace,chen2022qm,wang2024theoretical,ho2026flexible} are one important instance, providing mean forces and anisotropic, spatially inhomogeneous covariance estimates that are often largest near defects and transition states.  For crystalline-defect calculations, adaptive QM/MM coupling and elastic far-field analysis provide complementary ways to control spatial and modeling errors~\cite{wang2021posteriori,olson2023elastic}.

This creates a constrained-algorithm design problem, not only a modeling problem.  The \emph{geometry} of NEB and Dimer is the structure that makes them saddle-search methods rather than generic descent methods: NEB separates normal physical forces from tangential image redistribution, while Dimer couples a curvature-direction solve with a reflected-gradient translation.  These projections, tangent spaces, reflections, and zero sets encode the deterministic saddle problem.  A covariance preconditioner that is harmless for unconstrained gradient descent can be harmful here, because it need not commute with the normal--tangential decomposition and may move the stationary set from the minimum-energy path of the mean potential to a metric-distorted one.  Thus covariance should inform the step geometry without changing the equations whose solution is being sought.

Existing uses of uncertainty in saddle-search computations mainly act outside this update geometry.  GP-NEB constructs a probabilistic path surrogate~\cite{koistinen2017nudged,koistinen2020}, and active-learning force-field algorithms use uncertainty to select configurations for high-fidelity labeling~\cite{podryabinkin2017active,vandermause2020fly,kulichenko2023uncertainty}; force uncertainty can also come from deep ensembles~\cite{lakshminarayanan2017simple}, Bayesian potentials~\cite{frederiksen2004bayesian}, conformal or calibration-based UQ frameworks~\cite{ho2026flexible,yu2025conformal}, or atomistic uncertainty frameworks~\cite{perez2025uncertainty}.  More broadly, randomized collocation and least-squares polynomial-chaos constructions~\cite{tang2014discrete,guo2017stochastic,jakeman2017generalized,guo2019data} and information-theoretic sensitivity bounds for stochastic dynamics~\cite{dupuis2016path,tsourtis2015parametric} provide complementary UQ tools for propagating or ranking uncertainty.  These mechanisms are essential for reducing model bias, but they do not by themselves make each NEB or Dimer step direction-dependent in the local reliability of the force.

The distinction from GP-NEB is one of numerical setting rather than only implementation.  GP-NEB is most natural when high-fidelity force evaluations are scarce and the Gaussian-process posterior itself is the path model.  The present methods assume that a stochastic, ensemble, or learned force model already supplies repeated force queries together with a calibrated covariance estimate, and ask how that covariance should enter the constrained optimizer without changing the mean-potential target equations.  Direct comparison with GP-NEB therefore depends on the surrogate class, training set, and cost of reference-force acquisition; here we isolate the optimizer-level covariance geometry.

Our design principle is to let covariance change the step metric, but not the stationarity equations defining the mean-potential saddle problem.  For NEB, with $G=(\SigF+\lambda\Id)^{-1}$, we replace the Euclidean normal projection by the oblique projection
\[
  Q_{\perp,G}z=z-G\tau\,\frac{\tau^\top z}{\tau^\top G\tau},
\]
so that $Q_{\perp,G}G\nabla\barE=0$ is equivalent to the classical MEP condition $\nabla\barE\parallel\tau$, not to a metric-shifted equation.  For Dimer, whose reflected-gradient translation is full rank, the metric preconditions the reflected gradient and noisy rotational residual while preserving the critical-point set.  We call the algorithms uncertainty-aware NEB (UA-NEB) and uncertainty-aware Dimer (UA-Dimer).

We make three contributions.  First, we derive covariance-weighted UA-NEB and UA-Dimer updates that preserve the saddle-search targets.  Second, we cast the iterations as Robbins--Monro recursions, prove local convergence under stated stability assumptions, and give an explicit canonical Lyapunov verification for UA-NEB.  Third, we test the mechanism under controlled covariance models on an analytic saddle-search problem and a $127$-atom bcc tungsten vacancy hop, using paired seeds and matched force-evaluation counts.

The rest of the paper is organized as follows.  Section~\ref{sec:problem} fixes the deterministic and stochastic problem setting; Section~\ref{sec:algorithms} develops the UA-NEB and UA-Dimer algorithms together with the climbing and active-learning variants; Section~\ref{sec:convergence} establishes the local convergence theory and the scalable covariance realizations; Section~\ref{sec:numerics} reports the numerical experiments; Section~\ref{sec:conclusion} concludes.

\smallskip
\noindent\textbf{Notation.}\ Throughout, $\ip{\cdot}{\cdot}$ and $\norm{\cdot}$ denote the Euclidean inner product and norm, and $\Id$ denotes the identity matrix.  For a nonzero vector $a$, $aa^\top/\norm{a}^2$ is the Euclidean rank-one projection onto $\mathrm{span}\{a\}$ and $P_a:=\Id-aa^\top/\norm{a}^2$ is the corresponding orthogonal complement projection; in particular, for a unit Dimer direction $v$, $P_v=\Id-vv^\top$.  Expectations and variances are taken with respect to the randomness in the force and covariance queries made up to the current iteration.  Metric-dependent projections, weighted norms, residuals, and NEB/Dimer spacings are introduced where they first appear, in \S\ref{sec:algorithms} and \S\ref{sec:convergence}.

\section{Problem setting}
\label{sec:problem}

We work in a mass-weighted, or otherwise preconditioned, configuration
$x\in\R^d$ after removing rigid translations, fixed atoms, and imposed
linear constraints.  Two energies are kept distinct.  The high-fidelity
potential $E_\star$ defines the physical barrier, while $\barE$ denotes the
posterior, ensemble, or surrogate mean energy seen by the optimizer.  The
algorithms in this paper seek saddles and minimum-energy paths of $\barE$;
covariance estimates do not define a new target energy, but instead quantify
which force directions are reliable enough to use in a numerical step.

For later comparison with physical barriers, harmonic transition-state theory
\cite{vineyard1957frequency} gives
$k(T)=\nu^\dagger(T)\exp[-(E_\star(x_\star^\dagger)-E_\star(a))/k_BT]$, where
$x_\star^\dagger$ is the high-fidelity saddle and $\nu^\dagger$ is a
Hessian-eigenvalue prefactor.  A barrier error $\delta$
therefore shifts the rate by
\begin{equation}
  k_\delta(T)/k(T)=\exp[-\delta/k_BT],
  \label{eq:rate_sensitivity}
\end{equation}
so a $10$ meV error gives a rate shift of about $21\%$ at $600$ K and
$47\%$ at $300$ K.  The numerical experiments therefore report barrier
errors directly, rather than only residual norms.  We next recall the
deterministic NEB and Dimer equations, then specify the stochastic force and
covariance interface.

\subsection{Deterministic saddle-search background}
\label{sec:deterministic}

Let $E:\R^d\to\R$ be a smooth deterministic potential and let $a,b$ be two
local minimizers.  In the algorithms below this deterministic target is
$E=\barE$.  A minimum energy path (MEP) is characterized by the
vanishing of the force normal to the path.  In the continuum notation,
\(\varphi:[0,1]\to\R^d\), with $\varphi(0)=a$ and $\varphi(1)=b$, satisfies
\begin{equation}
  (\Id-\tau\tau^\top)\nabla E(\varphi(s))=0,
  \qquad
  \tau(s)=\frac{\varphi'(s)}{\norm{\varphi'(s)}} ,
  \label{eq:mep_condition}
\end{equation}
away from critical points.  A highest point on a generic MEP is an index-one saddle $x^\dagger$, $\nabla E(x^\dagger)=0$, with exactly one negative Hessian eigenvalue.
Stability of the continuous MEP and convergence of discrete MEP
approximations have recently been analyzed in
\cite{liu2024stability,liu2022convergence}.

\subsubsection{NEB}
NEB~\cite{jonsson1998neb,henkelman2000improved} discretizes a path from $a$ to $b$ by fixed endpoints $x_0=a$, $x_{n+1}=b$, and interior images $x_1,\ldots,x_n$.  For an interior image, $\tau_i$ denotes the discrete tangent \emph{estimator}.  We use the energy-weighted rule of \cite{henkelman2000improved}.  Under this rule, the tangent is chosen from the forward and backward secants
$\Delta_i^\pm=x_{i\pm1}-x_i$, with an energy-weighted combination near local
energy extrema, and then unit-normalized.  We denote this standard mapping by
\begin{equation}
  \tau_i=\mathcal T_i(x_{i-1},x_i,x_{i+1};E).
  \label{eq:nudged_tangent_rule}
\end{equation}
This choice avoids the common corner-cutting and sliding-down instabilities
of the original tangent rule; in the small-spacing limit on a smooth MEP,
$\tau_i\to\tau(s)$ of \eqref{eq:mep_condition}.

The deterministic NEB force on an interior image is
\begin{equation}
  \cF^{\rm NEB}_i
  =
  -(\Id-\tau_i\tau_i^\top)\nabla E(x_i)
  + k_s\big(\norm{x_{i+1}-x_i}-\norm{x_i-x_{i-1}}\big)\tau_i ,
  \label{eq:classical_neb}
\end{equation}
The normal component of the physical force relaxes the band onto the MEP,
while the spring component distributes images along the path.  A
climbing-image modification removes the spring force and reverses the
tangential true force on the image with largest energy
\cite{henkelman2000climbing}.
Preconditioned MEP finders accelerate this deterministic relaxation by changing
the path metric~\cite{makri2019preconditioning}; the covariance metric below
has the different role of damping uncertain force directions while preserving
\eqref{eq:mep_condition}.

The stochastic algorithms below preserve the MEP stationarity condition while
replacing the Euclidean normal-force projection in \eqref{eq:classical_neb} by
covariance-weighted directions.

\subsubsection{Dimer}
The Dimer method~\cite{henkelman1999dimer,gould2016dimer} seeks an index-one saddle without constructing an entire path.  Given a center $x$ and unit direction $v$, it estimates the Hessian action by a centered force difference.  With the energy-Hessian convention,
\begin{equation}
  \widehat H_h(x)v
  =
  -\frac{\Fhat(x+hv)-\Fhat(x-hv)}{2h}.
  \label{eq:hvp}
\end{equation}
Here $\Fhat$ denotes the queried force; in the deterministic case
$\Fhat=F=-\nabla E$, while in the stochastic setting it is the oracle force
specified in the next subsection.
The orientation is rotated toward the lowest-curvature eigenvector by approximately minimizing the Rayleigh quotient $v^\top \nabla^2 E(x)v$ over $\norm{v}=1$.  The center then moves along the reflected gradient
\begin{equation}
  r_x(x,v)
  =
  -\nabla E(x)+2vv^\top\nabla E(x),
  \label{eq:dimer_force}
\end{equation}
which descends in directions orthogonal to $v$ and ascends along $v$.  Since
independent force noise makes the variance of \eqref{eq:hvp} scale like
$h^{-2}$, the rotational step requires its own covariance weighting.  Thus in
both NEB and Dimer, covariance may change the metric of a step but not the
deterministic equations that define the target MEP or saddle.
Related saddle-search algorithms exploit MEP geometry, preconditioning,
high-index saddle dynamics, or problem-specific nullspaces in other ways;
examples include solution-landscape construction by generalized high-index
saddle dynamics~\cite{yin2021solutionlandscape}, spring-pair dynamics guided
by the MEP tangent, and nullspace-preserving saddle search for ordered phase
transitions with translational invariance~\cite{cui2024spring,cui2025efficient}.

\subsection{Stochastic force model and covariance estimation}
\label{sec:covariance}

The algorithms access $\barE$ through a stochastic force oracle.  For a
configuration $x\in\R^d$, the oracle returns
\begin{equation}
  \Fhat(x,\omega) = -\nabla \barE(x) + \zeta(x,\omega),
  \qquad
  \E[\zeta\mid x]=0,
  \qquad
  \SigF(x)=\Cov[\zeta\mid x],
  \label{eq:force_model}
\end{equation}
where $\omega$ labels model randomness, such as an ensemble member,
bootstrap replica, posterior draw, or Gaussian-process sample.  The
covariance $\SigF$ may be obtained from ensemble or Bayesian force models
\cite{frederiksen2004bayesian,lakshminarayanan2017simple,podryabinkin2017active,kulichenko2023uncertainty,perez2025uncertainty}.
Throughout we use the \emph{energy-Hessian sign convention}: the physical
force is $F=-\nabla E$, the gradient estimator is $\widehat g=-\widehat F$,
and the Hessian-vector product is $\nabla^2 E\,v$.  With this convention the
centered-difference Hessian-vector product carries the minus sign in
\eqref{eq:hvp}.  A force-Jacobian convention would omit this sign; it is not
used in the analysis.

The stochastic approximation target is the mean energy $\barE$.  The
high-fidelity energy $E_\star$ enters through calibration, validation, and
active learning, not through the deterministic drift of the optimizer.  Let
$\mathcal F_k$ be the filtration generated by all force, covariance, and
active-learning decisions up to iteration $k$.  For an adapted query $x_k$,
define
\begin{equation}
  \ghat(x_k,\omega_{k+1})=-\Fhat(x_k,\omega_{k+1})
  = \nabla \barE(x_k)-\zeta(x_k,\omega_{k+1}),
  \label{eq:gradient_estimator}
\end{equation}
so $\E[\ghat(x_k,\omega_{k+1})\mid\mathcal F_k]=\nabla\barE(x_k)$.

The algorithms below require only a calibrated covariance operator
consistent with this conditional-moment model.  Ensemble covariance is one
common way to build such an operator.  If ensemble energies are also
available, the uncertainty of a reported barrier can be monitored directly,
for example by $\Var_m[E^{(m)}(x_c)-E^{(m)}(a)]$, where $x_c$ is the
reported barrier configuration.  If a direct barrier variance is unavailable,
one can instead use a linearized propagation of the available covariance
information through the reported barrier functional; concrete operator-probing
estimators are described in Supplementary Section~SM5.  These checks affect reporting and
active-learning triggers, not the mean-potential target.

\subsubsection{Calibrated ensemble covariance}
Let $\{F^{(m)}(x)\}_{m=1}^M$ be force predictions from an independently seeded, bootstrapped, or posterior ensemble, and let
\begin{equation}
  \bar F_M(x)=\frac1M\sum_{m=1}^M F^{(m)}(x).
\end{equation}
The raw sample covariance is
\begin{equation}
  \widehat S_M(x)
  =
  \frac{1}{M-1}
  \sum_{m=1}^M
  \big(F^{(m)}(x)-\bar F_M(x)\big)
  \big(F^{(m)}(x)-\bar F_M(x)\big)^\top .
  \label{eq:raw_ensemble_cov}
\end{equation}
The calibrated covariance used by the algorithms for a single-member
stochastic query is
\begin{equation}
  \widehat\SigF_M(x)
  =
  s_{\rm cal}^2\widehat S_M(x)+\sigma_{\rm floor}^2 \Id ,
  \label{eq:ensemble_cov}
\end{equation}
where $s_{\rm cal}$ is a scalar or blockwise calibration factor obtained on a
validation set and $\sigma_{\rm floor}>0$ prevents spuriously zero variance.
If a force query averages several independent members, the covariance in
\eqref{eq:force_model} is replaced by the covariance of that averaged query;
sampling one member per iteration gives \eqref{eq:ensemble_cov} directly.

The raw ensemble covariance in \eqref{eq:raw_ensemble_cov} measures the
spread of model predictions.  Before it is interpreted as a force-error
covariance, it should be calibrated against high-fidelity validation forces.
We use the Gaussian negative-log-likelihood scaling $s_{\rm cal}$, or a
blockwise variant by atomic species, optionally combined with a
coverage-based safety criterion.  Finite ensembles can also produce noisy
eigenvectors; when needed, we control this effect by a shrinkage parameter
$\rho\in[0,1]$ that blends the raw sample covariance with a block-sparse
projection.  The full calibration likelihood, coverage condition, and
shrinkage formula are collected in Supplementary Section~SM5.  In the numerical experiments
below the prescribed covariance is exact by construction, so we set
$s_{\rm cal}=1$ and $\rho=0$.

\subsubsection{Operator realizations}
For large $d$, the algorithms should not require a dense matrix.  We assume the covariance
module exposes the following operations:
\begin{equation}
  z\mapsto \widehat\SigF(x)z,\qquad
  z\mapsto (\widehat\SigF(x)+\lambda\Id)^{-1}z,\qquad
  \log\det(\widehat\SigF(x)+\lambda\Id),
  \label{eq:covariance_interface}
\end{equation}
and, when the log-determinant penalty is active, directional derivatives in
arbitrary directions $u$,
\[
  u^\top\nabla_x\log\det(\widehat\SigF(x)+\lambda\Id).
\]
The inverse operation can be exact for small blocks, a Woodbury apply for low
rank, or a few Krylov iterations for sparse local covariances.  This operator
formulation is sufficient for the algorithms below, since the NEB and Dimer
updates require only products of $G$ with force-like vectors and scalar
products such as $\tau^\top Gz$.

For a local or message-passing force model with finite effective interaction radius,
epistemic uncertainty is often local in atomic environments even when the
force itself is many-body.  This motivates block or low-rank operator models
such as
\begin{equation}
  \SigF(x)
  \approx
  \sum_{\ell=1}^{L} P_\ell^\top B_\ell(x)P_\ell,
  \label{eq:local_block_cov}
\end{equation}
where $P_\ell:\R^d\to\R^{b_\ell}$ extracts the Cartesian force components of a local atom cluster or defect-core patch, and $B_\ell\succeq0$ is a small dense block.  Overlapping blocks are allowed.  A related low-rank form follows from local feature-gradient Jacobians:
\begin{equation}
  \SigF(x)\approx J_\theta(x) C_\theta J_\theta(x)^\top
  = U(x)C(x)U(x)^\top ,
  \label{eq:jacobian_cov}
\end{equation}
where $C_\theta$ is a posterior or ensemble covariance in parameter or
latent-feature space.  These forms are introduced here only as covariance
interfaces for the algorithms.  Detailed cost models and Woodbury or Krylov
applies are deferred to \S\ref{sec:complexity} and
Supplementary Section~SM4; energy-only probing is described in
Supplementary Section~SM5; and the decomposition of barrier error into
stochastic, optimization, and model terms is stated in Supplementary
Proposition~SM2.1.

\section{Algorithms}
\label{sec:algorithms}

The covariance interface of \S\ref{sec:problem} gives directional
reliability information, but it does not by itself say how that information
should enter a constrained saddle-search algorithm.  The design constraint in
this section is therefore geometric: covariance may change the metric of a
stochastic step, but it must not change the deterministic stationarity
equations for the mean potential.

For NEB, the difficulty is a rank-deficient normal projection that does not
commute with a generic covariance preconditioner.  We resolve this with an
oblique projection whose zero set is the classical MEP condition, then check
local stability and the effect of a metric spring.  For Dimer, the projection
obstruction disappears, but the noisy Hessian-vector rotation and the
reflected-gradient translation require separate covariance weights.  The two
constructions are different, but the organizing principle is the same: use
uncertainty as step geometry, not as a new target energy.
Throughout this section, $\epsilon>0$ denotes a fixed small denominator
regularization.

\subsection{Uncertainty-aware NEB}
\label{sec:uaneb}

Let $x_i$ be an interior image.  Given a regularization parameter $\lambda>0$, define the reliability metric
\begin{equation}
  G_i=G(x_i):=(\SigF(x_i)+\lambda\Id)^{-1}.
  \label{eq:metric}
\end{equation}
Directions with high force uncertainty have smaller weight in the $G_i$ metric and are damped in the preconditioned force.  Since $(\lambda+\lambda_{\max}(\SigF(x_i)))^{-1}\Id\preceq G_i\preceq \lambda^{-1}\Id$, the regularization $\lambda$ keeps the inverse metric uniformly bounded even when the covariance estimator is rank deficient.

\subsubsection{Weighted tangent and projections}

The tangent component must be removed without shifting the MEP equation.  For a gradient vector $g$, the Euclidean projection of $G_i g$ and the $G_i$-orthogonal projection both impose stationarity on a metric-shifted vector.  We instead use the rank-$(d-1)$ projection whose range is the Euclidean normal subspace $\{z:\tau_i^\top z=0\}$ and whose kernel is the metric tangent direction:
\begin{equation}
  Q_{\perp,G_i}(z)
  =
  z
  -
  G_i\tau_i\,
  \frac{\tau_i^\top z}{\tau_i^\top G_i\tau_i},
  \qquad
  Q_{\parallel,G_i}(z)
  =
  G_i\tau_i\,
  \frac{\tau_i^\top z}{\tau_i^\top G_i\tau_i}.
  \label{eq:oblique_projection}
\end{equation}
This is an \emph{oblique} projection: its range is the Euclidean tangent-orthogonal hyperplane and its kernel is $\mathrm{span}\{G_i\tau_i\}$.  Lemma~\ref{lem:oblique_projection} records the zero-set identity needed below.

\begin{lemma}[MEP-preserving projection identities]
\label{lem:oblique_projection}
Let $G$ be symmetric positive definite and $\tau\ne0$.  The operators in \eqref{eq:oblique_projection} satisfy
\[
  Q_{\perp,G}^2=Q_{\perp,G},\qquad
  Q_{\parallel,G}^2=Q_{\parallel,G},\qquad
  \tau^\top Q_{\perp,G}z=0,
\]
and $z=Q_{\perp,G}z+Q_{\parallel,G}z$.  Moreover,
\begin{equation}
  Q_{\perp,G}Gg=0
  \quad\Longleftrightarrow\quad
  (\Id-\tau\tau^\top/\norm{\tau}^2)g=0 .
  \label{eq:mep_zero_set}
\end{equation}
\end{lemma}

\begin{proof}
$\tau^\top Q_{\perp,G}z=\tau^\top z-\tau^\top G\tau\cdot\tau^\top z/(\tau^\top G\tau)=0$.  Idempotence follows since $Q_{\perp,G}^2 z=Q_{\perp,G}z-G\tau\cdot(\tau^\top Q_{\perp,G}z)/(\tau^\top G\tau)=Q_{\perp,G}z$, and analogously for $Q_{\parallel,G}$.  For \eqref{eq:mep_zero_set}: if $Q_{\perp,G}Gg=0$, then $Gg=G\tau\,(\tau^\top Gg)/(\tau^\top G\tau)$, so $g\parallel\tau$; conversely, $g=c\tau$ gives $Q_{\perp,G}Gg=cG\tau-cG\tau=0$.
\end{proof}

\begin{remark}[Why the alternative projections fail]
\label{rem:alt_projections}
Euclidean projection of $Gg$ and $G$-orthogonal projection both require $Gg\parallel\tau$, equivalently $g\parallel G^{-1}\tau$.  Thus their zero set is a metric-shifted line unless $\tau$ is a $G$-eigenvector.  A concrete $2\times2$ example is already decisive: take
\[
  \tau=e_1,\qquad
  G=
  \begin{bmatrix}
  7/4 & -3\sqrt3/4\\
  -3\sqrt3/4 & 13/4
  \end{bmatrix}.
\]
The classical zero set contains $g=e_1$, but
$(I-\tau\tau^\top)Gg=(0,-3\sqrt3/4)^\top\ne0$.  The competing projections
instead select $g\parallel G^{-1}\tau\propto(13,3\sqrt3)^\top$, which is not
parallel to $\tau$.  The oblique form \eqref{eq:oblique_projection} is used
because $Q_{\perp,G}Gg=0$ is equivalent to the classical MEP condition
$g\parallel\tau$.  The extended comparison in
Supplementary Section~SM1, including Supplementary
Figure~S1, visualizes the corresponding flow
geometry.
\end{remark}

The metric also admits a constrained natural-gradient interpretation.  Given a gradient sample $g$, the covariance-preconditioned normal force $-Q_{\perp,G}Gg$ is the solution of
\begin{align}
  -Q_{\perp,G}Gg
  &=
  \argmin_{\tau^\top s=0}
  \left\{
    g^\top s+\frac12 s^\top G^{-1}s
  \right\},
  \notag\\[-1mm]
  &=
  \argmin_{\tau^\top s=0}
  \norm{s+Gg}_{G^{-1}}^{2}.
  \label{eq:neb_variational_step}
\end{align}
Thus high-variance directions are damped before projection, while the deterministic MEP stationarity condition is unchanged.  After computing $Gg$ and $G\tau$, the projection in \eqref{eq:oblique_projection} requires only the scalar products $\tau^\top Gg$ and $\tau^\top G\tau$.

Lemma~\ref{lem:oblique_projection} is algebraic; the next statement gives the corresponding local dynamics.  With fixed tangent, the oblique-force flow has the restricted energy as a Lyapunov function and is locally asymptotically stable at the constrained minimizer.

\begin{lemma}[Fixed-tangent local stability]
\label{lem:fixed_tangent_stability}
Fix a nonzero tangent $\tau$ and a symmetric positive definite metric $G$.  Consider the constrained deterministic flow
\[
  \dot x=-Q_{\perp,G}G\nabla\barE(x),
  \qquad
  \tau^\top(x-x_\star)=0,
\]
where $x_\star$ is a nondegenerate constrained minimizer of $\barE$ on the affine hyperplane $\tau^\top(x-x_\star)=0$.  Then $x_\star$ is a locally asymptotically stable equilibrium of the constrained flow.  In particular, replacing the Euclidean normal force by the oblique covariance-weighted normal force changes the local metric and time scale, but not the constrained critical point.
\end{lemma}

\begin{proof}
The vector field is tangent to the hyperplane because $\tau^\top Q_{\perp,G}z=0$ for every $z$.  If $s=-Q_{\perp,G}G\nabla\barE(x)$, then the variational identity \eqref{eq:neb_variational_step} gives the KKT relation $G^{-1}s+\nabla\barE(x)+\mu\tau=0$ for some scalar $\mu$.  Since $\tau^\top s=0$,
\[
  \nabla\barE(x)^\top s=-s^\top G^{-1}s\le0 .
\]
Thus the restricted energy is a Lyapunov function.  Near a nondegenerate constrained minimizer its restriction to the hyperplane is locally strongly convex, so the dissipation above gives local asymptotic stability.
\end{proof}

Lemma~\ref{lem:fixed_tangent_stability} holds $\tau$ fixed.  In the running iteration $\tau_i$ evolves with the band; the next lemma gives the local regularity needed to treat this as smooth state dependence.  The optional tangent relaxation enters only as a summable perturbation.

\begin{lemma}[Tangent regularity and slow-variable contribution]
\label{lem:tangent_slow}
Let $X_\star=(x_{1,\star},\ldots,x_{n,\star})$ be a nondegenerate MEP discretization at which the energy ordering at each interior image is strict, so that the branch of the Henkelman--J\'onsson tangent mapping \eqref{eq:nudged_tangent_rule} is locally constant.  Then $X\mapsto\tau_i(X)$ is $C^1$ and locally Lipschitz in a neighborhood $\mathcal K_\star$ of $X_\star$.  If the relaxed tangent defined below in \eqref{eq:tangent_smoothing} is used with $\omega_{\tau,k}\to0$ and $\sum_k\alpha_k\omega_{\tau,k}<\infty$, then the difference between the relaxed tangent and the instantaneous Henkelman--J\'onsson tangent contributes only a summable perturbation to the recursion \eqref{eq:sa}.
\end{lemma}
This local regularity claim is verified in Supplementary Section~SM2.1.  Thus the instantaneous tangent is a
smooth local function of the band, and the optional tangent relaxation is
absorbed by the bias term $b_k$ in \eqref{eq:sa}.

Lemmas~\ref{lem:oblique_projection} and~\ref{lem:fixed_tangent_stability} settle the normal-force part of the NEB iteration.  The spring force controls image spacing and is measured in the same metric:
\begin{equation}
  F_{i}^{\rm spring}
  =
  k_s
  \big(
    \norm{x_{i+1}-x_i}_{G_i}
    -
    \norm{x_i-x_{i-1}}_{G_i}
  \big)
  \frac{\tau_i}{\sqrt{\tau_i^\top G_i\tau_i}},
  \label{eq:g_spring}
\end{equation}
where $\norm{z}_{G_i}=(z^\top G_i z)^{1/2}$.  We use the image metric $G_i$ on both adjacent segments, which avoids midpoint covariances and keeps Algorithm~\ref{alg:uaneb} image-parallel.  The frozen-metric spring mismatch is a tangential discretization error; under coupled refinement it enters the SA bias when $\sum_k\alpha_k h_{\max,k}^2<\infty$ (see Supplementary Section~SM2.3).

\subsubsection{UA-NEB force}

With the normal projection and spring term specified, let $\ghat(x_i)$ be the noisy gradient estimate of $\nabla\barE(x_i)$.  The UA-NEB force is
\begin{equation}
  \cF_i^{\rm UA}
  =
  -Q_{\perp,G_i}\,G_i\,\ghat(x_i)
  +F_i^{\rm spring}.
  \label{eq:ua_neb_force}
\end{equation}
The first term is path-normal in the Euclidean sense; the second redistributes images along the path.  If $\SigF(x_i)=\sigma^2\Id$, then $G_i$ is a scalar multiple of $\Id$, the oblique projection becomes the Euclidean normal projection, and \eqref{eq:ua_neb_force} reduces to classical stochastic NEB up to scalar step-size and spring-stiffness rescaling.  Thus \eqref{eq:ua_neb_force} is the unpenalized UA-NEB drift.

For finite iterations one may add a transient log-determinant penalty
\begin{equation}
  \Psi_\lambda(x)=\log\det(\SigF(x)+\lambda\Id).
  \label{eq:logdet_penalty}
\end{equation}
When $\SigF$ is differentiable,
$\partial_{x_j}\Psi_\lambda(x)=\tr[(\SigF(x)+\lambda\Id)^{-1}\partial_{x_j}\SigF(x)]$.
The derivative can be evaluated by automatic differentiation, local finite differences, or the local blocks in \eqref{eq:local_block_cov}.  The image update is
\begin{equation}
  x_i^{k+1}
  =
  x_i^k
  +\alpha_k \cF_i^{{\rm UA},k}
  -\alpha_k\gamma_k \nabla_x\Psi_\lambda(x_i^k),
  \qquad i=1,\ldots,n.
  \label{eq:ua_neb_update}
\end{equation}
For convergence to the unpenalized MEP of $\barE$, we take $\gamma_k\downarrow0$ with $\sum_k\alpha_k\gamma_k<\infty$.  A nonzero limiting $\gamma$ instead defines the regularized target $\barE+\gamma\Psi_\lambda$.

\subsubsection{Climbing image and tangent smoothing}

After the band has relaxed close to an MEP, the highest-energy image can be converted into a UA climbing image.  Let $c$ maximize either $\barE(x_i)$ or the lower-confidence score $\barE(x_i)-\kappa_E\widehat\sigma_E(x_i)$, where $\widehat\sigma_E(x_i)$ is the ensemble standard deviation of $E^{(m)}(x_i)-E^{(m)}(a)$.  Barrier reports may use the corresponding upper-confidence value $\barE(x_c)+\kappa_E\widehat\sigma_E(x_c)-\barE(a)$.  The UA climbing force is
\begin{equation}
  \cF_c^{\rm climb}
  =
  -G_c\ghat(x_c)
  +2Q_{\parallel,G_c}G_c\ghat(x_c).
  \label{eq:ua_climbing}
\end{equation}
No spring force is applied to the climbing image.  The optional log-determinant penalty still enters through \eqref{eq:ua_neb_update}.  The remaining images continue to use \eqref{eq:ua_neb_update}; in noisy runs, climbing is activated only after the top image is stable under the chosen energy score.

The next lemma checks the noiseless limit against the classical climbing-image dynamics.

\begin{lemma}[Classical-limit consistency of the UA climbing image]
\label{lem:climbing_limit}
Under the step-size reparametrization $\widetilde\alpha_k:=\lambda^{-1}\alpha_k$ and the limit $\SigF\to0$, $G_c\to\lambda^{-1}\Id$ and $Q_{\parallel,G_c}\to\tau_c\tau_c^\top/\norm{\tau_c}^2$, so the deterministic UA climbing-image update reduces to the classical climbing-image iteration of \cite{henkelman2000climbing} with effective step size $\widetilde\alpha_k$.
\end{lemma}
\begin{proof}
$G_c=(\SigF+\lambda\Id)^{-1}\to\lambda^{-1}\Id$ as $\SigF\to 0$.  Substituting into \eqref{eq:oblique_projection} gives $Q_{\parallel,G}z=\tau\tau^\top z/\norm{\tau}^2$, so \eqref{eq:ua_climbing} reduces to $\lambda^{-1}[-\nabla\barE+2(\tau\tau^\top/\norm{\tau}^2)\nabla\barE]$; rescaling $\alpha_k\to\widetilde\alpha_k$ absorbs the $\lambda^{-1}$.
\end{proof}

To reduce tangent noise we use the relaxed tangent
\begin{equation}
  \widetilde\tau_i^k
  =
  \frac{(1-\omega_{\tau,k})\tau_i^{k,{\rm raw}}
  +\omega_{\tau,k}\widetilde\tau_i^{k-1}}
  {\left\|(1-\omega_{\tau,k})\tau_i^{k,{\rm raw}}
  +\omega_{\tau,k}\widetilde\tau_i^{k-1}\right\|},
  \qquad 0\le\omega_{\tau,k}<1,
  \label{eq:tangent_smoothing}
\end{equation}
where $\tau_i^{k,{\rm raw}}$ is the energy-weighted NEB tangent
\eqref{eq:nudged_tangent_rule}.  In the convergence analysis
$\omega_{\tau,k}\to0$ and
$\sum_k\alpha_k\omega_{\tau,k}<\infty$ make the relaxation a transient
stabilization rather than a change of the limiting tangent rule.

\begin{algorithm}[t]
\caption{UA-NEB}
\label{alg:uaneb}
\begin{algorithmic}[1]
\Require Initial band, covariance estimator, and algorithm parameters.
\Ensure Approximate MEP and saddle candidate.
\For{$k=0,1,2,\ldots$}
  \State Query ensemble forces/covariances at all images; apply $G_i^k=(\widehat\SigF(x_i^k)+\lambda\Id)^{-1}$ as an operator.
  \State Compute energy-weighted tangents $\tau_i^k$ with optional smoothing \eqref{eq:tangent_smoothing}.
  \For{$i=1,\ldots,n$ in parallel}
    \State Apply \eqref{eq:oblique_projection} and \eqref{eq:g_spring} to assemble $\cF_i^{{\rm UA},k}$; use \eqref{eq:ua_climbing} at the climbing image once the climbing criterion is satisfied.
    \State Set $q_i^k=\nabla_x\log\det(\widehat\SigF(x_i^k)+\lambda\Id)$ if $\gamma_k>0$ and $q_i^k=0$ otherwise; set $s_i^k=\alpha_k(\cF_i^{{\rm UA},k}-\gamma_k q_i^k)$.
  \EndFor
  \State Apply trust-radius scaling to $\{s_i^k\}$ if needed; update interior images, keep endpoints fixed, and reparametrize when spacing ratios exceed tolerance.
  \State Stop when \eqref{eq:neb_residual} and the maximum covariance score meet tolerances.
\EndFor
\end{algorithmic}
\end{algorithm}

An optional active-learning trigger can combine pathwise covariance magnitude and
directional force noise.  Let
$\cF_i^k$ denote the current UA-NEB force and
$d_i^k=\cF_i^k/(\norm{\cF_i^k}+\epsilon)$ the corresponding regularized search direction.  We use, for example,
\begin{equation}
  \max_i \lambda_{\max}(\widehat\SigF(x_i^k))>\eta_{\rm var},
  \quad
  \max_i \frac{\sqrt{(d_i^k)^\top\widehat\SigF(x_i^k)d_i^k}}
              {\norm{\cF_i^k}+\epsilon}>\eta_{\rm rel},
  \quad
  \widehat{\Var}[\Delta E_{\rm barrier}]>\eta_{\rm bar}^2.
  \label{eq:al_trigger}
\end{equation}
Here $\widehat{\Var}$ is the empirical ensemble variance of the barrier estimate.  The first condition detects under-sampled configurations, the second detects search directions dominated by model noise, and the third controls the barrier estimate.
For Dimer, the same criteria are applied to the current center and Dimer
endpoints, with the reflected-gradient and rotational residuals replacing the
path-image force in the reliability ratio.

\begin{remark}[Scope of the active-learning trigger]
Equation~\eqref{eq:al_trigger} is optional; Section~\ref{sec:numerics} instead uses a matched label-refresh comparator.  Its thresholds are set by force-error calibration and the desired barrier tolerance.
\end{remark}

\subsection{Uncertainty-aware Dimer}
\label{sec:uadimer}

The Dimer method couples two numerical tasks with different noise profiles.
The rotation of $v$ uses a centered force difference, whose variance scales as
$h^{-2}$ in the Dimer length $h$ (not the NEB image spacing $h_i^\pm$); the
translation of $x$ uses the gradient at the Dimer center.  Treating these two
queries with the same scalar uncertainty would miss the dominant source of
rotational noise.

The geometry is also different from NEB.  The Dimer translation uses the
full-rank reflection $-\Id+2vv^\top$, so a covariance metric cannot displace a
rank-deficient zero set in the way it can for NEB.  Instead, the metric damps
unreliable components of the reflected-gradient translation and weights the
noisy Hessian-vector residual used for rotation.  We first construct the
Hessian-vector covariance and Dimer length, then define the weighted rotation
and translation.  Throughout $\norm{v_k}=1$, with $P_v$ as defined in the
notation.

\subsubsection{Hessian-vector covariance and Dimer length}
\label{sec:hvp_var}

For each ensemble member $m$, define
\begin{equation}
  H_h^{(m)}(x)v
  =
  -\frac{F^{(m)}(x+hv)-F^{(m)}(x-hv)}{2h}.
  \label{eq:member_hvp}
\end{equation}
The ensemble mean gives $\widehat H_h(x)v$ and the sample covariance gives
\begin{equation}
  \widehat\SigH(x,v)
  =
  s_H^2
  \cdot
  \frac{1}{M-1}
  \sum_{m=1}^M
  \big(H_h^{(m)}(x)v-\widehat H_h(x)v\big)
  \big(H_h^{(m)}(x)v-\widehat H_h(x)v\big)^\top
  +\sigma_{H,{\rm floor}}^2\Id .
  \label{eq:hvp_cov}
\end{equation}
Here $s_H$ is calibrated on reference Hessian-vector products by the same negative-log-likelihood scaling as $s_{\rm cal}$ in \eqref{eq:ensemble_cov} (see Supplementary Section~SM5), and $\sigma_{H,{\rm floor}}>0$ is a variance floor.  In MLIP applications, these products can be obtained by centered finite differences of reference forces at $x\pm hv$ (for example DFT or a trusted classical potential), or by direct reference Hessian-vector evaluations when available.
If only force covariance matrices are available, and the force evaluations at
$x\pm hv$ are treated as conditionally independent, then
\begin{equation}
  \SigH(x,v)
  \approx
  \frac{\SigF(x+hv)+\SigF(x-hv)}{4h^2}.
  \label{eq:hvp_cov_force}
\end{equation}
The exact paired variance formula is given in Supplementary Section~SM2.2.  It specializes to
\eqref{eq:hvp_cov_force} at zero correlation, while positive paired-error
correlation makes \eqref{eq:hvp_cov_force} conservative.

This covariance also determines the Dimer length.  If $\barE\in C^4$, the centered-difference bias in \eqref{eq:hvp} is $O(h^2)$, while independent force noise produces $O(h^{-2})$ variance.  We choose the smallest $h$ satisfying
\begin{equation}
  \frac{\tr(P_v\widehat\SigH(x,v)P_v)}
       {\norm{P_v\widehat H_h(x)v}^2+\epsilon}
  \le \eta_H ,
  \label{eq:h_adapt}
\end{equation}
subject to $h_{\min}\le h\le h_{\max}$.

\subsubsection{Weighted rotation and translation}
\label{sec:weighted_rot}

With a rotational covariance from \eqref{eq:hvp_cov} or
\eqref{eq:hvp_cov_force} in hand, we can now weight the rotation step in the
same way \S\ref{sec:uaneb} weighted the NEB normal force.  The deterministic
Dimer orientation residual is
\begin{equation}
  r_v(x,v)=P_v \nabla^2\barE(x)v .
  \label{eq:orientation_residual}
\end{equation}
It vanishes when $v$ is an eigenvector; the local Dimer branch of interest is
the lowest-curvature eigenvector selected by Rayleigh-quotient descent.  We
replace $\nabla^2\barE(x)v$ by $\widehat H_h(x)v$ and precondition the tangent
residual with
\begin{equation}
  C_v(x,v)=(P_v\SigH(x,v)P_v+\lambda_H P_v)^\dagger,
  \label{eq:rotation_metric}
\end{equation}
where the pseudoinverse is taken on the tangent space $T_v\mathbb S^{d-1}$.  A retracted rotation step is
\begin{equation}
  \widetilde v_{k+1}
  =
  v_k-\beta_k C_{v_k}(x_k,v_k)P_{v_k}\widehat H_h(x_k)v_k,
  \qquad
  v_{k+1}=\frac{\widetilde v_{k+1}}{\norm{\widetilde v_{k+1}}}.
  \label{eq:weighted_rotation}
\end{equation}
The parameter $\lambda_H>0$ bounds the angular gain; in noisy runs we also cap the rotation by a trust angle $\theta_{\max}$.

For translation, the covariance metric acts on the reflected gradient.  Unlike NEB, the Dimer translation uses a full-rank reflection rather than a rank-deficient projection.  The required deterministic facts are preservation of critical points and local stability once the orientation has found the unstable mode.  With $G(x)=(\SigF(x)+\lambda\Id)^{-1}$, the UA-Dimer translation is
\begin{equation}
  x_{k+1}
  =
  x_k+\alpha_k
  G(x_k)
  \left[-\ghat(x_k)
  +2v_{k+1}v_{k+1}^\top\ghat(x_k)\right].
  \label{eq:ua_dimer_translation}
\end{equation}
As in NEB, a transient term $-\alpha_k\gamma_k\nabla\Psi_\lambda(x_k)$ can be added.  The reflection $vv^\top$ remains Euclidean; the covariance metric preconditions only the reflected gradient.  The next lemma records the two deterministic facts used later.

\begin{lemma}[Metric-preconditioned Dimer translation preserves critical points]
\label{lem:dimer_zero_set}
Let $G\succ0$, $\norm{v}=1$, and $R_v=-\Id+2vv^\top$.  Define $\mathcal T_G(x,v)=G R_v\nabla\barE(x)$.  Then
\[
  \mathcal T_G(x,v)=0
  \quad\Longleftrightarrow\quad
  \nabla\barE(x)=0 .
\]
Moreover, if $x^\dagger$ is a nondegenerate index-one saddle, $v^\dagger$ is the unstable Hessian eigenvector, and $H^\dagger=\nabla^2\barE(x^\dagger)$ is positive definite on $(v^\dagger)^\perp$, then the frozen-orientation linearization
\[
  \dot e=G(x^\dagger)R_{v^\dagger}H^\dagger e
\]
has all eigenvalues in the open left half-plane.
\end{lemma}

\begin{proof}
$G$ is invertible and $R_v$ is an involution, hence $\mathcal T_G=0$ iff $R_v\nabla\barE=0$ iff $\nabla\barE=0$.  At the saddle with $v=v^\dagger$, write $H^\dagger v^\dagger=\lambda_1 v^\dagger$ with $\lambda_1<0$ and $H^\dagger|_{(v^\dagger)^\perp}\succ0$.  Then $R_{v^\dagger}H^\dagger$ is symmetric negative definite: it keeps the negative curvature in the $v^\dagger$ direction and reverses the positive curvatures on the orthogonal subspace.  Since $G^{1/2}(R_{v^\dagger}H^\dagger)G^{1/2}$ is also symmetric negative definite and $G R_{v^\dagger}H^\dagger$ is similar to it, the linearization has real negative eigenvalues.
\end{proof}
Thus the Dimer metric preconditions the reflected-gradient direction; it is not a NEB-type zero-set correction.

In Algorithm~\ref{alg:uadimer}, the rotational covariance score is
\[
  \widehat\sigma_{H,k}^2
  =
  \tr(P_{v_k}\widehat\SigH(x_k,v_k)P_{v_k}).
\]

\begin{algorithm}[t]
\caption{UA-Dimer}
\label{alg:uadimer}
\begin{algorithmic}[1]
\Require $x_0,v_0$ with $\norm{v_0}=1$; Dimer length $h$; covariance estimator; parameters $\lambda,\lambda_H,\{\alpha_k\},\{\beta_k\},\Delta_k,\theta_{\max},\eta_H$.
\Ensure Approximate saddle $x^\dagger$ and unstable direction $v^\dagger$.
\For{$k=0,1,2,\ldots$}
  \State Query paired ensemble forces at $x_k\pm hv_k$; form $\widehat H_h(x_k)v_k$ and $\widehat\SigH(x_k,v_k)$ by \eqref{eq:hvp_cov} or \eqref{eq:hvp_cov_force}.
  \State If the HVP-noise ratio \eqref{eq:h_adapt} exceeds $\eta_H$, enlarge $h$ or fall back to $C_{v_k}=P_{v_k}$ and request reference labels at $x_k\pm hv_k$.
  \State Set $\widehat r_v^k=P_{v_k}\widehat H_h(x_k)v_k$, $\delta v_k=-\beta_k C_{v_k}\widehat r_v^k$, capped at $\theta_{\max}$; retract $v_{k+1}=(v_k+\delta v_k)/\norm{v_k+\delta v_k}$.
  \State Query center force/covariance; apply the metric reflected-gradient translation, trust-radius scale it, and update $x_{k+1}$.
  \State Request reference labels when \eqref{eq:al_trigger} holds or $\widehat\sigma_{H,k}^2$ is high; stop by gradient, rotation-residual, and covariance tolerances.
\EndFor
\end{algorithmic}
\end{algorithm}

\subsubsection{NEB--Dimer handoff}
\label{sec:neb_dimer_handoff}

UA-NEB identifies the transition channel, while UA-Dimer refines a local saddle candidate.

Let $c$ be the final climbing or highest-energy image, initialize $x_0^{\rm D}=x_c^{\rm NEB}$, and set the initial Dimer direction to the local path tangent $v_0^{\rm D}=\tau_c^{\rm NEB}$.
The handoff is accepted only if the path-normal residual is small relative to its uncertainty,
\begin{equation}
  \frac{
    \norm{Q_{\perp,G_c}G_c\ghat(x_c)}
  }{
    \sqrt{
      \tr(Q_{\perp,G_c}G_c\widehat\SigF(x_c)G_cQ_{\perp,G_c}^\top)
    }+\epsilon
  }
  \le \eta_{\rm hand}.
  \label{eq:handoff_condition}
\end{equation}
If \eqref{eq:handoff_condition} fails, the band relaxation continues or reference labels are requested near $x_c$.  Upon acceptance, the Dimer phase normalizes $v_0^{\rm D}$ and sets
\[
  \Delta_0^{\rm D}
  =
  \min\{\Delta_{\rm NEB},2\norm{x_c^{\rm NEB}-x_{c-1}^{\rm NEB}}\},
\]
where $\Delta_{\rm NEB}$ is the final NEB trust radius.  The non-climbing band is retained only as path context.

\begin{remark}[Scope and calibration of the handoff test]
Condition~\eqref{eq:handoff_condition} is a signal-to-noise test for seeding local refinement.  By itself it does not imply membership in the Dimer stability neighborhood; the Dimer convergence statement applies once the local hypotheses of Section~\ref{sec:convergence} hold.
\end{remark}

Inexact covariance solves, trust-region tests, stopping criteria, and cached local blocks are treated as implementation perturbations.  In Section~\ref{sec:convergence} they enter the bias term $b_k$; concrete tolerances and merit functions are recorded in Supplementary Section~SM3.

\section{Theory: local stochastic stability and scalable covariance}
\label{sec:convergence}

Section~\ref{sec:algorithms} fixed the deterministic geometry: in their
mean-field limits, UA-NEB and UA-Dimer preserve the classical stationary set.
The question here is whether this geometry survives stochastic implementation.
We prove a local result: once the mean drift dissipates a residual, stochastic
forces, covariance estimation, metric-solve errors, finite Dimer differences,
and transient penalties enter as martingale noise plus summable bias, and the
mean-potential residual converges to zero.

For canonical UA-NEB the required dissipative drift is verified by an explicit
Lyapunov function.  For UA-Dimer we identify the correct local residual near the
selected index-one branch; convergence then follows from the same stochastic
argument when the analogous local Lyapunov drift holds.  Thus the theorem is a
residual-level stability result.  The analytic experiment checks its
finite-time \(O(1/k)\) residual prediction, while the barrier-error experiments
assess whether this residual-level mechanism improves the reported saddle
barrier at fixed iteration counts.

\subsection{Residuals and stochastic-approximation form}

Let $X_k$ denote the full algorithmic state: for NEB,
$X_k=(x_1^k,\ldots,x_n^k)\in\R^{nd}$, while for Dimer,
$X_k=(x_k,v_k)\in \R^d\times\mathbb S^{d-1}$.  After projection to local
coordinates on the sphere in the Dimer case, both algorithms can be written as
\begin{equation}
  X_{k+1}=X_k+\alpha_k[h(X_k)+M_{k+1}+b_k].
  \label{eq:sa}
\end{equation}
Here $h$ is the deterministic drift obtained by replacing stochastic forces and
covariances by their conditional expectations.  The term $M_{k+1}$ is a
martingale difference, and $b_k$ collects finite-difference bias,
covariance-estimation bias, regularization error, retraction error, and the
transient log-determinant penalty when $\gamma_k$ is summable.

For NEB, the $i$th component of $h$ has the form
\begin{equation}
  h_i^{\rm NEB}(X)
  =
  -Q_{\perp,G_i(X)}G_i(X)\nabla\barE(x_i)
  +F_i^{\rm spring}(X),
  \label{eq:neb_drift}
\end{equation}
where $G_i$ is built from the limiting calibrated covariance.  When the same
force sample feeds both the gradient and the covariance, the noise splits into
a force martingale, a covariance fluctuation, and a finite-ensemble bias term;
independent ensemble splits or slowly updated calibration models make the
martingale terms conditionally mean zero up to the bias collected in $b_k$.

The Dimer branch has the same SA structure, now on
$\R^d\times\mathbb S^{d-1}$.  To avoid confusion with the finite-difference
Dimer length $h$, denote the Dimer deterministic vector field by $d^{\rm D}$.
In a local chart, with a fixed rotational-to-translational stepsize ratio
$\rho_\beta$, the rotation and translation stack into
\begin{equation}
  d^{\rm D}(x,v)=
  \begin{bmatrix}
    G(x)r_x(x,v)\\
    -\rho_\beta C_v(x,v)r_v(x,v)
  \end{bmatrix},
  \quad
  \begin{aligned}
    r_x(x,v)&=-\nabla\barE(x)+2vv^\top\nabla\barE(x),\\
    r_v(x,v)&=(\Id-vv^\top)\nabla^2\barE(x)v,
  \end{aligned}
  \label{eq:dimer_drift}
\end{equation}
The implemented Dimer update differs from this drift by
$O(h_{\rm dim}^2)$ finite-difference bias and second-order retraction error
bounded by the trust angle; in \eqref{eq:sa}, $h=d^{\rm D}$ for the Dimer
branch.

The residuals below are the quantities controlled by the theorem.  For NEB, we use
\begin{equation}
  \cR_{\rm NEB}(X)
  =
  \sum_{i=1}^n
  \norm{Q_{\perp,G_i}G_i\nabla\barE(x_i)}^2
  +
  \rho_s
  \sum_{i=1}^n
  \left|
    \norm{x_{i+1}-x_i}_{G_i}
    -
    \norm{x_i-x_{i-1}}_{G_i}
  \right|^2 ,
  \label{eq:neb_residual}
\end{equation}
with endpoints fixed.  A climbing-image residual additionally includes the
reflected tangential force on the current highest-energy image.  For Dimer, the
residual is
\begin{equation}
  \cR_{\rm D}(x,v)
  =
  \norm{(\Id-vv^\top)\nabla^2\barE(x)v}^2
  +
  \norm{-\nabla\barE(x)+2vv^\top\nabla\barE(x)}^2 .
  \label{eq:dimer_residual}
\end{equation}
The residual is the stationarity measure; its local relation to barrier-error
reporting is separated in Supplementary Proposition~SM2.1.

For Dimer, the orientation component of \eqref{eq:dimer_residual} vanishes at
any Hessian eigenvector, so the local branch is selected by the negative mode.
If $x^\dagger$ is a nondegenerate index-one saddle, $v^\dagger$ is the
normalized negative eigenvector of $\nabla^2\barE(x^\dagger)$, and the remaining
eigenvalues have a positive spectral gap, then a standard perturbation argument
gives, in a sufficiently small neighborhood of
$(x^\dagger,\pm v^\dagger)$,
\begin{equation}
  \label{eq:dimer_residual_equiv}
  \begin{aligned}
  c_{\rm D}^-\,
  \dist\!\left((x,v),\{(x^\dagger,\pm v^\dagger)\}\right)^2
  &\le
  \cR_{\rm D}(x,v)  \\
  &\le
  c_{\rm D}^+\,
  \dist\!\left((x,v),\{(x^\dagger,\pm v^\dagger)\}\right)^2 .
  \end{aligned}
\end{equation}
Thus $\cR_{\rm D}$ is the correct local residual once the Dimer phase has
selected the unstable branch.  To see why, write $v$ in a local chart
$v=\pm v^\dagger+\eta+O(\norm{\eta}^2)$ with
$\eta\perp v^\dagger$.  The reflected-gradient component linearizes as
$R_{v^\dagger}\nabla^2\barE(x^\dagger)(x-x^\dagger)+O(\norm{x-x^\dagger}^2+\norm{\eta}\norm{x-x^\dagger})$,
whose leading matrix is invertible because the saddle is nondegenerate.  The
orientation component linearizes in $\eta$ as
$(\nabla^2\barE(x^\dagger)-\lambda_1 I)\eta$ plus terms of order
$O(\norm{x-x^\dagger})$, and the spectral gap between the negative eigenvalue
$\lambda_1$ and the remaining eigenvalues makes this angular block invertible.
The stacked residual map therefore has an invertible block-triangular
linearization, modulo the sign symmetry $v\sim -v$; \eqref{eq:dimer_residual_equiv}
then follows from the inverse function theorem.  Further details are given in
Supplementary Section~SM2.4.

\subsection{Local stability and stochastic convergence}

The convergence proof needs one local deterministic input and one stochastic
input: the mean drift must dissipate the residual, and the remaining terms in
\eqref{eq:sa} must be small in the stochastic-approximation sense.  We collect
these requirements in a single setting.

\begin{assumption}[Local stochastic stability setting]
\label{ass:smooth}
Let $\cR$ denote the relevant squared residual, either
\eqref{eq:neb_residual} or \eqref{eq:dimer_residual}.  After a possible first
entrance time and re-indexing, the iterates remain in a compact neighborhood
$\mathcal K$ of the target MEP discretization or saddle.  On $\mathcal K$,
$\barE\in C^3$, $\SigF\in C^1$, the covariance eigenvalues are bounded, and the
deterministic drift $h$ is locally Lipschitz.  For the finite-difference Dimer
bias statement below, assume in addition $\barE\in C^4$.

With $\mathcal S=\{X\in\mathcal K:\cR(X)=0\}$, there exists a $C^1$ Lyapunov
function $V$ and constants $c_1,c_2,c_3>0$ such that
\begin{equation}
  c_1\cR(X)\le V(X)\le c_2\cR(X),
  \qquad
  \nabla V(X)\cdot h(X)\le -c_3\cR(X)
  \label{eq:lyapunov_drift}
\end{equation}
on $\mathcal K$.

The noise in \eqref{eq:sa} satisfies
\[
  \E[M_{k+1}\mid\mathcal F_k]=0,
  \qquad
  \E[\norm{M_{k+1}}^2\mid\mathcal F_k]\le C_M(1+\norm{X_k}^2),
\]
and the bias is summable:
\begin{equation}
  \sum_{k=0}^\infty \alpha_k\norm{b_k}<\infty
  \qquad\hbox{almost surely.}
  \label{eq:bias_summable}
\end{equation}
The translational stepsizes satisfy
\[
  \alpha_k>0,\qquad
  \sum_{k=0}^\infty\alpha_k=\infty,\qquad
  \sum_{k=0}^\infty\alpha_k^2<\infty .
\]
For Dimer, $\beta_k=\rho_\beta\alpha_k$ with fixed $\rho_\beta>0$, up to a
summable deviation.  Its finite-difference bias is $O(h_k^2)$ when
$\barE\in C^4$, so \eqref{eq:bias_summable} follows from
$\sum_k\alpha_k h_k^2<\infty$.  A summable log-determinant penalty is included
in $b_k$ through the condition $\sum_k\alpha_k\gamma_k<\infty$.
\end{assumption}

For canonical UA-NEB, the local Lyapunov condition can be verified explicitly.

\paragraph{Canonical UA-NEB verification.}
Let $X_\star$ be a nondegenerate discretized MEP.  For each image, let
$\tau_{i,\star}$ be the limiting energy-weighted tangent and set
\[
  H_i:=\nabla^2\barE(x_{i,\star})|_{\tau_{i,\star}^\perp}\succ\mu_H\Id,
  \qquad
  G_{i,\star}:=(\SigF(x_{i,\star})+\lambda\Id)^{-1}.
\]
With
$y_i=Q_{\perp,G_{i,\star}}(x_i-x_{i,\star})$ and
$z_i=Q_{\parallel,G_{i,\star}}(x_i-x_{i,\star})$, take
\begin{equation}
  \begin{aligned}
  V(X)
  &=
  \sum_{i=1}^n
  \left[
    \tfrac12 y_i^\top H_i y_i+\tfrac{\theta}{2}\norm{z_i}^2
  \right]  \\
  &\quad
  +\tfrac{\rho_s}{2}
  \sum_{i=1}^n
  \left(
    \norm{x_{i+1}-x_i}-\norm{x_i-x_{i-1}}
  \right)^2 .
  \end{aligned}
  \label{eq:V_canonical}
\end{equation}

\begin{proposition}[Canonical UA-NEB stability]
\label{lem:lyapunov_neb_canonical}
For suitable $\theta,\rho_s>0$, in a sufficiently small neighborhood
$\mathcal K_\star$ of $X_\star$, the deterministic UA-NEB drift
$h^{\rm NEB}$ in \eqref{eq:neb_drift} satisfies
\[
  c_1\cR_{\rm NEB}(X)\le V(X)\le c_2\cR_{\rm NEB}(X),
  \qquad
  \nabla V(X)\cdot h^{\rm NEB}(X)\le -c_3\,\cR_{\rm NEB}(X),
\]
with constants depending on the local smoothness and spectral bounds, the
constrained Hessian gap, tangent-branch separation, image-spacing lower bounds,
spring stiffness, and the size of $\mathcal K_\star$.  Hence the deterministic
part of Assumption~\ref{ass:smooth} holds for this canonical UA-NEB setting.
\end{proposition}
Here is the drift mechanism.  Freeze $\tau_i$ and $G_i$ at $X_\star$ and
decompose $x_i-x_{i,\star}=y_i+z_i$ into the Euclidean normal and metric
tangent blocks used in \eqref{eq:V_canonical}.  On the normal subspace,
Taylor expansion gives
$\nabla\barE(x_i)=H_i y_i+O(\norm{X-X_\star}^2)$ after removing the tangential
MEP component.  The variational identity \eqref{eq:neb_variational_step}
implies, for any normal vector $w$,
\[
  w^\top Q_{\perp,G_{i,\star}}G_{i,\star}w
  \ge c_G\norm{w}^2 ,
\]
where $c_G>0$ is the Schur-complement lower bound of $G_{i,\star}$ on
$\tau_{i,\star}^{\perp}$.  Hence the normal contribution satisfies
\[
  \nabla_{y_i}V\cdot
  \big(-Q_{\perp,G_{i,\star}}G_{i,\star}\nabla\barE(x_i)\big)
  \le
  -c\,\norm{H_i y_i}^2+O(\norm{X-X_\star}^3).
\]
The longitudinal variables are controlled by the linearized spring-spacing
operator; choosing $\theta$ and $\rho_s$ balances the normal--tangential cross
terms, giving a negative definite frozen linearization.  The Lyapunov spring
block uses Euclidean spacings, but it is equivalent to the metric-spring
residual in \eqref{eq:neb_residual} because $G_i$ and $G_i^{-1}$ have uniformly
bounded spectra.  Smooth variation of $\tau_i(X)$, $G_i(X)$, and the metric
spring is then absorbed by shrinking $\mathcal K_\star$.  The full perturbation
estimates are recorded in Supplementary Section~SM2.5.

The nondegeneracy hypothesis on $X_\star$ is generic when $\barE$ admits an
isolated smooth continuous MEP with strictly positive constrained Hessian gap;
see \cite{ren2013climbing,weinan2002string} and
Supplementary Section~SM2.5.  We work in this canonical setting for
the explicit NEB stability verification.  For UA-Dimer,
\eqref{eq:dimer_residual_equiv} identifies the local residual; the same
stochastic conclusion applies when the Lyapunov drift in
Assumption~\ref{ass:smooth} holds on that branch.

Assumption~\ref{ass:smooth} gives the single estimate on which the stochastic
proof rests.  A Taylor expansion of $V$ at $X_k$, the
martingale property of $M_{k+1}$, the Lyapunov drift
\eqref{eq:lyapunov_drift}, and the second-moment bound on $\mathcal K$ imply
that, for sufficiently small $\sup_k\alpha_k$, there exist constants $c,C>0$
and a summable nonnegative sequence $\epsilon_k$ such that
\begin{equation}
  \E[V(X_{k+1})\mid\mathcal F_k]
  \le
  V(X_k)-c\alpha_k\cR(X_k)+C\alpha_k^2+\alpha_k\epsilon_k .
  \label{eq:rs_estimate}
\end{equation}
The Taylor-remainder estimates behind \eqref{eq:rs_estimate} are collected in
Supplementary Section~SM2.

\begin{theorem}[Local convergence]
\label{thm:convergence}
Suppose Assumption~\ref{ass:smooth} holds for the recursion in $\mathcal K$.
Then
\[
  \sum_{k=0}^\infty \alpha_k\cR(X_k)<\infty
  \qquad\hbox{and}\qquad
  \cR(X_k)\to0
  \quad\hbox{almost surely.}
\]
Moreover, if $\mathcal S\cap\mathcal K$ consists of isolated equilibria, then
$X_k$ converges almost surely to one of them.
\end{theorem}

\begin{proof}
With
$Y_k=V(X_k)+\sum_{j\ge k}C\alpha_j^2+\sum_{j\ge k}\alpha_j\epsilon_j$,
\eqref{eq:rs_estimate} gives
$\E[Y_{k+1}\mid\mathcal F_k]\le Y_k-c\alpha_k\cR(X_k)$.  The
Robbins--Siegmund theorem \cite{robbins1971convergence} therefore gives
convergence of $V(X_k)$ and summability of $\sum_k\alpha_k\cR(X_k)$.  The ODE
method for stochastic approximation with square-summable martingale noise and
summable bias \cite[Thm.~5.2.1]{kushner2003stochastic}
(see also \cite[Thm.~2.1]{borkar2008stochastic}) identifies the almost-sure
limit set with an internally chain-transitive set of $\dot X=h(X)$ inside
$\mathcal K$.  The strict Lyapunov drift \eqref{eq:lyapunov_drift} excludes
such sets outside $\mathcal S$, hence $\cR(X_k)\to0$ almost surely.  If
$\mathcal S\cap\mathcal K$ consists of isolated equilibria, the limit component
is a single point.
\end{proof}

If $\gamma_k\equiv\gamma>0$, the same proof applies after absorbing the
log-determinant term into the deterministic drift.  The limiting equations are
then those of the regularized landscape $\barE+\gamma\Psi_\lambda$ for NEB, or
the corresponding regularized reflected-gradient dynamics for Dimer.

\subsection{Local mean-square rate}
Theorem~\ref{thm:convergence} is qualitative.  If the deterministic drift is
locally linearly contractive in the Lyapunov function, the same estimate yields
a non-asymptotic $L^2$ rate.  In the canonical UA-NEB setting,
Proposition~\ref{lem:lyapunov_neb_canonical} gives this contraction directly:
since
$V\le c_2\cR_{\rm NEB}$ and
$\nabla V\cdot h^{\rm NEB}\le -c_3\cR_{\rm NEB}$, one may take
$\mu=c_3/c_2$.

\begin{proposition}[Local $L^2$ convergence rate]
\label{prop:rate}
Assume the setting of Theorem~\ref{thm:convergence}.  Let
$\mathcal K_\star\subset\mathcal K$ be a compact neighborhood of the target set
on which the stronger drift inequality
$\nabla V(X)\cdot h(X)\le -\mu V(X)$ holds for some $\mu>0$.  Let
$\tau_\star=\inf\{j\ge0:X_j\notin\mathcal K_\star\}$.  Choose
$\alpha_k=\alpha_0/(k+k_0)$ with $\alpha_0\mu>1$ and $k_0$ large enough that
$\alpha_k\mu\le 1/2$.  If the remainder sequence in \eqref{eq:rs_estimate}
satisfies $\epsilon_k=O((k+k_0)^{-q})$ for some $q>1$, then
\[
  \E\!\left[V(X_k)\mathbf 1_{\{\tau_\star>k\}}\right]=O\!\left(\frac{1}{k+k_0}\right)
\]
and, by the equivalence of $V$ and $\cR$,
\[
  \E[\cR(X_k)\mathbf 1_{\{\tau_\star>k\}}]=O(1/k).
\]
If the iteration is localized so that $\tau_\star=\infty$ almost surely, this is
the unconditional $L^2$ rate.
\end{proposition}

The proof combines \eqref{eq:rs_estimate} with the contraction
$\nabla V\cdot h\le -\mu V$ and a discrete Gronwall iteration
\cite[Ch.~2]{kushner2003stochastic}; the full product bound and bias estimates
are recorded in Supplementary Section~SM2.7.

Proposition~\ref{prop:rate} is the rate statement checked in the analytic
experiment: under the canonical setting verified in
Proposition~\ref{lem:lyapunov_neb_canonical}, UA-NEB has a localized $O(1/k)$
mean-square residual bound.  The barrier-error experiments in
Section~\ref{sec:numerics} then test whether this residual-level mechanism
improves the final reporting functional.
For UA-Dimer, the analogous rate statement requires the local Dimer Lyapunov
condition in Assumption~\ref{ass:smooth} and the strengthened contraction
hypothesis above.

Theorem~\ref{thm:convergence} is stated for a fixed mean force and covariance
model; finitely many retraining updates are handled by applying it after the
last update, and continuing retraining requires summable drift perturbations
(Supplementary Section~SM2.8).

\subsection{Scalable covariance realizations}
\label{sec:complexity}

The convergence theorem uses $\SigF$ only through spectral bounds and products
with $G=(\SigF+\lambda\Id)^{-1}$.  It therefore does not require dense
covariance matrices.  The algorithms use covariance through the operator
interface $z\mapsto Gz$ and scalar products involving $Gz$.

This interface is compatible with the standard scalable realizations used for
stochastic force models.  Dense covariance is useful for small validation
problems but costs $O(d^2)$ storage and $O(d^3)$ factorization.  Diagonal,
atomwise, and local-block covariances reduce storage to $O(d)$ or $O(Nb^2)$ and
make the covariance-weighted step linear in the number of local environments up
to the block or Krylov cost.

Low-rank forms $\SigF\approx UCU^\top$ can be applied by Woodbury with cost
$O(dr+r^3)$ when the small $r\times r$ solve is factored on demand, or
$O(dr+r^2)$ per apply when that factorization is cached.  Crystalline-defect
settings can additionally eliminate elastic far-field variables when uncertainty
is localized near the defect core.  Detailed per-iteration cost models and the
low-rank and far-field formulas are collected in
Supplementary Section~SM4.

\section{Numerical experiments}
\label{sec:numerics}

The experiments follow one question: does covariance help when it is placed in
the constrained update geometry?  On a controlled analytic MEP, we separate
metric weighting from scalar penalties and label refresh, and then check that
the observed residual decay is consistent with Proposition~\ref{prop:rate}.  A
covariance-rotation sweep asks when full tensor information matters.  A Dimer
test isolates local saddle refinement, and a W-vacancy benchmark tests path
search in an atomistic defect geometry.  The active-learning trigger
\eqref{eq:al_trigger} and the NEB--Dimer handoff condition
\eqref{eq:handoff_condition} are not varied here.

Unless stated otherwise, $\pm$ values are standard errors of the mean over
paired stochastic seeds.  Paired NEB variants share image count, iteration
count, trust radii, force-evaluation counts, initial band, and force-noise
sequence; the Dimer test shares initial centers, orientations, and force-query
counts.  Common NEB settings are identical across methods, and concrete
parameter values are collected in Supplementary Section~SM6 and Supplementary
Table~S4.

The analytic benchmark potential is
\begin{equation}
  E(x_1,x_2)
  =
  (x_1^2-1)^2
  +
  k\big(x_2-a(1-x_1^2)\big)^2,
  \qquad
  a=0.38,\quad k=7.5,
  \label{eq:toy_energy}
\end{equation}
whose exact MEP is $x_2=a(1-x_1^2)$ and whose saddle barrier is one.  We perturb the exact gradient by mean-zero Gaussian force noise whose covariance tube is placed near the transition region and whose largest eigenvector is nearly transverse to the path.  This setting is simple enough to interpret geometrically and anisotropic enough to expose path wandering in Euclidean stochastic NEB.

\subsection{Uncertainty-aware NEB on an analytic MEP}
\label{sec:numerics_neb}

The finite-step NEB experiment asks where uncertainty must enter the algorithm
to reduce barrier error while monitoring path deviation.  Figure~\ref{fig:ua_neb_benchmark}
compares six variants:
\begin{itemize}
  \item \textbf{std}: standard stochastic NEB (Euclidean projection, no covariance);
  \item \textbf{pen.}: log-determinant penalty only \eqref{eq:logdet_penalty} without metric weighting;
  \item \textbf{AL}: a periodic label-refresh baseline that replaces noisy forces by exact ones at the highest-uncertainty images with matched image updates; this tests uncertainty-guided labeling as a comparator, not the three-condition trigger \eqref{eq:al_trigger};
  \item \textbf{metric}: oblique metric force \eqref{eq:oblique_projection} only, Euclidean spring;
  \item \textbf{diag}: UA-NEB with $\widehat\SigF$ replaced by $\diag(\widehat\SigF)$, a per-component-variance ablation that isolates diagonal information from off-diagonal covariance;
  \item \textbf{UA}: full UA-NEB update \eqref{eq:ua_neb_force} including the metric spring \eqref{eq:g_spring}.
\end{itemize}
In the finite-step experiments we rescale each inverse-covariance metric to have trace two.  This fixes the step scale; at a fixed image it leaves the oblique normal direction and its zero set unchanged.

These ablations separate uncertainty used outside the constrained update geometry
(\textbf{pen.}, \textbf{AL}) from reliability placed directly in the NEB step
(\textbf{metric}, \textbf{diag}, \textbf{UA}).  In this base cell the
high-uncertainty transition region is close enough to Cartesian alignment that
diagonal weighting is expected to be a strong comparator.  The rotated sweep
(Figure~\ref{fig:ua_neb_sensitivity}) and the W-vacancy benchmark
(\S\ref{sec:atomistic_validation}) are designed to break this alignment.

\begin{figure}[t]
\centering
\includegraphics[width=.96\textwidth]{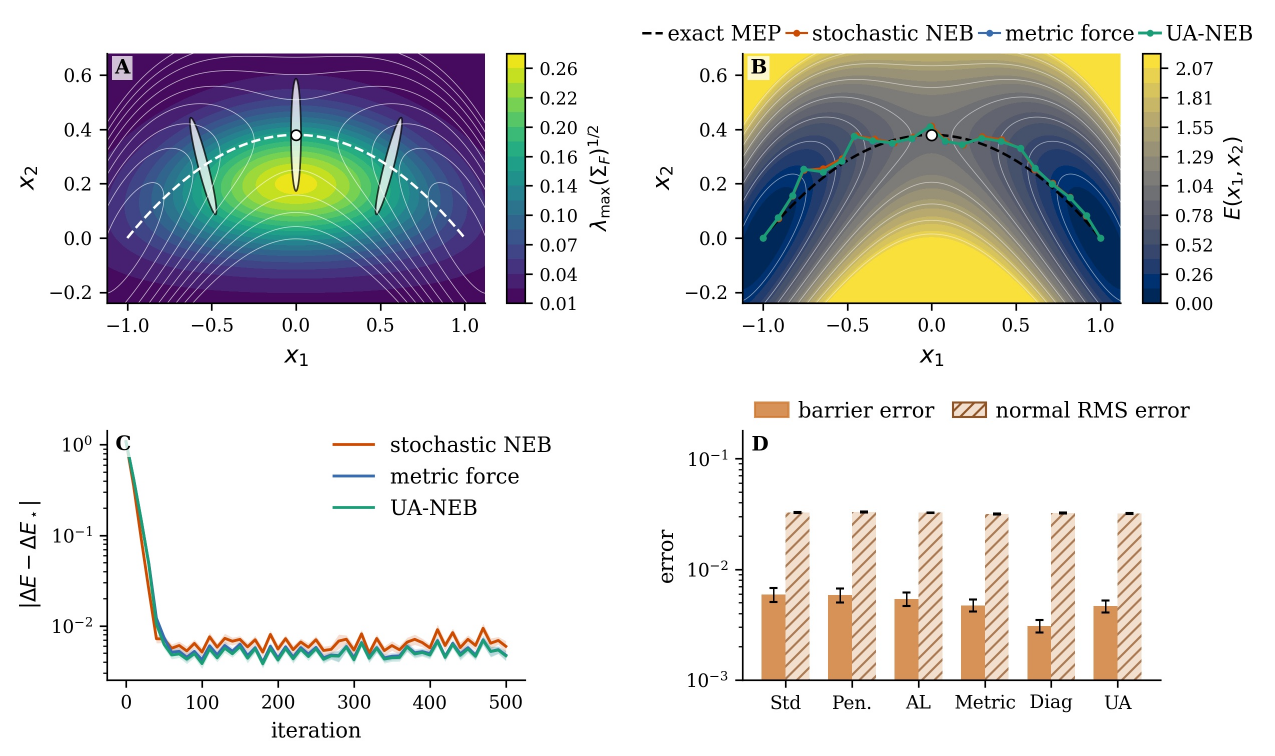}
\caption{Synthetic anisotropic-noise NEB benchmark.  (A) Largest force-covariance standard deviation; white ellipses show covariance eigenframes and the dashed curve is the exact MEP.  (B) Representative final paths for selected methods.  (C) Mean barrier-error trajectories over $200$ paired seeds, with standard-error bands.  (D) Final barrier error and normal-path RMS for all six ablations.  Metric-based updates provide the main barrier-error gain; diagonal weighting is unusually strong in this aligned base cell, motivating the rotated-covariance sweep and W-vacancy test.}
\label{fig:ua_neb_benchmark}
\end{figure}

The first conclusion is mechanistic.  Only the variants that put covariance into
the projected NEB step give a clear reduction in barrier error.  Full UA-NEB and
the metric-only ablation give nearly the same improvement (about one fifth,
paired Wilcoxon $p<4\!\times\!10^{-5}$), whereas the scalar penalty and periodic
label refresh do not change the outcome appreciably.  Diagonal weighting is even
stronger in this particular cell, reducing the mean barrier error by $48\%$,
because the high-uncertainty covariance eigenframe is close to the coordinate
axes.  Thus the base experiment identifies the active mechanism and also
explains why the diagonal realization must be carried as a necessary comparator.

The same analytic setting gives a residual-level consistency check for the
theory; the supporting diagnostic is reported in
Supplementary Section~SM7 and Supplementary
Figure~S2.  Over
$k\in[5,40]$, the fitted slopes range from $-1.43$ to $-1.26$, followed by a
finite-step plateau after $k\approx50$.  Its role is to check that the transient
residual scale used by the theory is visible before finite-step bias dominates.

The covariance-structure sweep answers the alignment question.  It rotates the
principal axes relative to the path-tangent/normal frame while varying the
perpendicular noise amplitude $\sigma_n^{(\mathrm{amp})}$.  Full UA-NEB improves
on stochastic NEB in 22 of the 35 cells and on diagonal weighting in 20 cells.
The cellwise pattern matters more than the aggregate count: diagonal weighting
remains strongest near the aligned high-anisotropy corner represented by the
base experiment, whereas the full tensor can gain value when reliable and
unreliable directions rotate with the path.  The W-vacancy benchmark below
targets this non-Cartesian regime.

\begin{figure}[t]
\centering
\includegraphics[width=.95\textwidth]{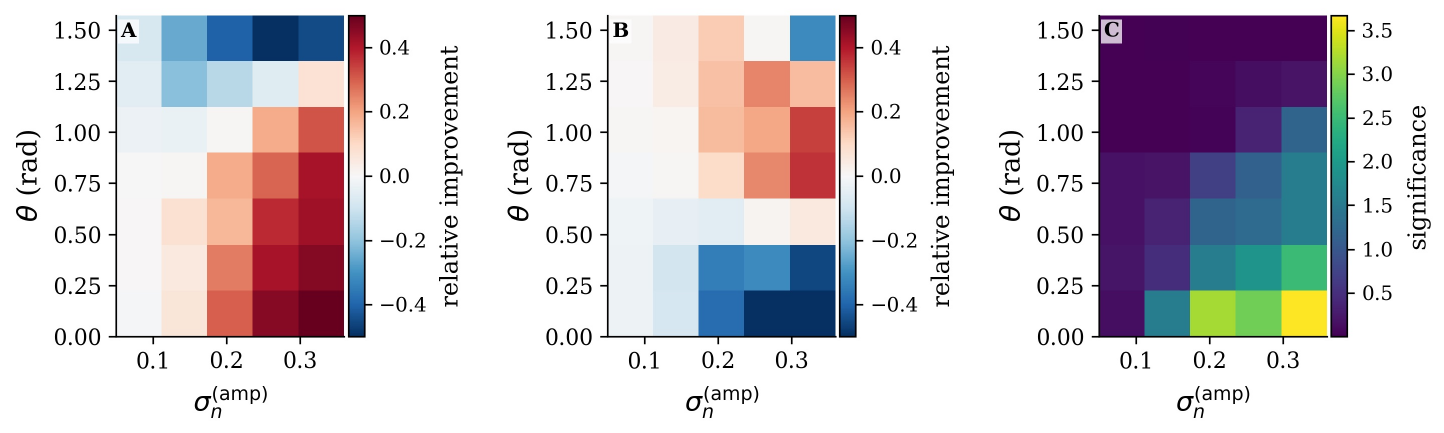}
\caption{Covariance-structure sensitivity sweep.  (A,B) Relative improvement of full UA-NEB over stochastic NEB and diagonal weighting; positive values mean smaller barrier error, and the color scale is clipped to the displayed range.  (C) Unadjusted one-sided Wilcoxon scores $-\log_{10}p$ for full UA $<$ std.}
\label{fig:ua_neb_sensitivity}
\end{figure}

\subsection{Dimer rotation and translation test}
\label{sec:numerics_dimer}

The Dimer experiment asks a narrower question: whether the same covariance
geometry lowers the local residual floor after a path method has supplied a
saddle candidate.  We initialize the center away from the saddle on the analytic
potential and corrupt the HVP by the same anisotropic force covariance.  UA-Dimer
uses the HVP covariance \eqref{eq:hvp_cov_force} in the rotation and the
normalized inverse force covariance in the reflected-gradient translation.
The supporting trajectories and residual distributions are shown in
Supplementary Section~SM8 and Supplementary
Figure~S3.

The result is a residual-floor improvement rather than a basin change.  The mean
final reflected-gradient residual drops from $0.224\pm0.012$ to
$0.174\pm0.009$ (paired Wilcoxon one-sided $p<10^{-12}$), while the
distance-to-saddle success rates are $96.0\%$ and $95.5\%$, respectively.  Covariance
weighting therefore improves local refinement before any active-learning
intervention is invoked.

This local test evaluates the covariance-weighted Dimer update as a refinement
mechanism for saddle candidates supplied by a path method; it does not tune the
handoff threshold $\eta_{\rm hand}$ in \eqref{eq:handoff_condition}.  We do not
report a separate atomistic Dimer experiment, since that would require a distinct
local saddle-refinement benchmark with controlled Hessian-vector noise.

\subsection{Atomistic W-vacancy benchmark}
\label{sec:atomistic_validation}

The atomistic test returns to path search under the non-Cartesian covariance
structure suggested by the sweep.  We test a nearest-neighbor monovacancy hop in
bcc tungsten using the Mason--Nguyen-Manh--Becquart EAM/FS potential
\cite{mason2017empirical} as both the conditional mean force and the
deterministic benchmark, with a prescribed mean-zero stochastic perturbation of
covariance form \eqref{eq:force_model}.  The covariance is localized near the
vacancy core, strongly transverse to the hop, and coupled to an elastic far
field.  The deterministic reference barrier is
$\Delta E_{\rm ref}\approx1.5379\,{\rm eV}$; the remaining numerical settings
are listed in Supplementary Section~SM6 and Supplementary
Table~S4.

\begin{figure}[t]
\centering
\includegraphics[width=.92\textwidth]{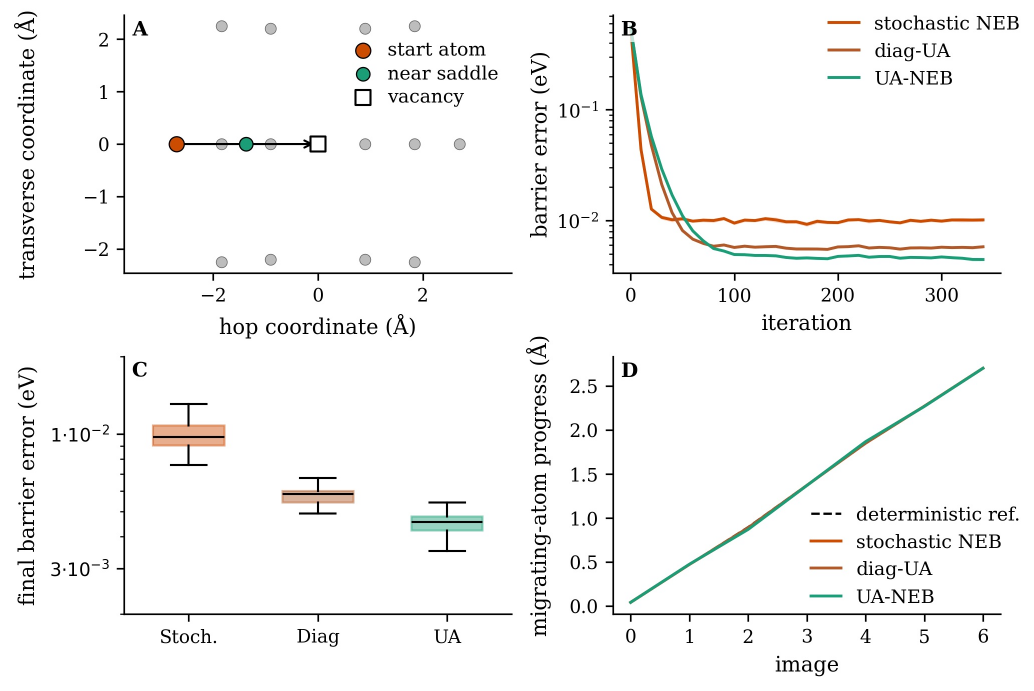}
\caption{Atomistic bcc tungsten vacancy-hop benchmark with EAM/FS mean force, prescribed anisotropic covariance, $N=127$, $n=7$, and $24$ seeds.  (A) Vacancy core and migrating atom.  (B) Mean barrier-error trajectory.  (C) Final barrier-error distribution.  (D) Migrating-atom progress.  Full UA-NEB reduces final barrier error by $\sim56\%$ relative to stochastic NEB and by $\sim23\%$ relative to diagonal weighting.}
\label{fig:w_vacancy_benchmark}
\end{figure}

Figure~\ref{fig:w_vacancy_benchmark} is the central atomistic test for full
covariance geometry.  Diagonal weighting already improves the barrier estimate, but the full
tensor adds the expected gain in a defect geometry whose reliable and unreliable
directions are not Cartesian.  The final mean absolute barrier errors are
$10.14\pm0.39$, $5.80\pm0.14$, and $4.45\pm0.14\,{\rm meV}$ for stochastic,
diagonal, and full UA-NEB.  Thus full UA-NEB reduces the mean error by $56\%$
relative to stochastic NEB and by $23\%$ relative to diagonal weighting.  All
$24$ paired differences have the same sign, so the one-sided Wilcoxon tests
reach their finite lower bound $p=6.0\!\times\!10^{-8}$; paired
Hodges--Lehmann improvements are $5.43$ and $1.34\,{\rm meV}$ for the two
full-UA comparisons.

The residual and rate implications are consistent with the barrier errors.  The
deterministic-EAM normal-force residual follows the same ordering, with full
UA-NEB lowest, while panel D shows the same vacancy-hop channel across methods.
Using \eqref{eq:rate_sensitivity}, the reduction in mean barrier error from
$10.14$ to $4.45\,{\rm meV}$ decreases the corresponding absolute rate-factor
error at $600\,{\rm K}$ from about $18\%$ to about $8\%$.

\section{Conclusion}
\label{sec:conclusion}

This paper treats anisotropic force uncertainty as part of numerical algorithm design for constrained saddle search.  Rather than using covariance only to decide where to refine a surrogate, UA-NEB and UA-Dimer use it as a local metric for stochastic steps while preserving the deterministic NEB and Dimer stationarity equations.  This is the sense in which the methods are geometry-preserving: uncertainty changes how an update is taken, not which mean-potential MEP or index-one saddle is targeted.

The analysis casts the iterations as Robbins--Monro recursions with controlled metric-solve bias, proves local almost-sure convergence, and gives an $O(1/k)$ $L^2$ residual rate under a strengthened local contraction condition.  The experiments support the same algorithmic message: metric weighting, not variance avoidance alone, drives the main gains in the analytic tests, and the W-vacancy benchmark shows that full covariance information can improve over both stochastic NEB and diagonal weighting in an atomistic defect calculation.  These tests isolate optimizer geometry using controlled covariance models; in deployed MLIP workflows, covariance calibration and reference-label policy remain part of the modeling pipeline.  Because the metric can be realized through covariance-vector products, structured local blocks, or low-rank reductions, the approach can be coupled to ensemble or surrogate force models without changing the underlying NEB/Dimer stationary sets.

\bibliographystyle{siamplain}
\bibliography{bib}

\end{document}